\documentclass[11pt,a4paper]{amsart}
\usepackage{amsmath,amssymb,amsthm}
\usepackage[utf8]{inputenc}
\usepackage{graphicx}
\usepackage{color}

\swapnumbers

\theoremstyle{plain}
\newtheorem{thm}[equation]{Theorem}
\newtheorem{lemma}[equation]{Lemma}

\newtheorem{cor}[equation]{Corollary}
\theoremstyle{definition}
\newtheorem{defn}[equation]{Definition}

\theoremstyle{remark}
\newtheorem{rem}[equation]{Remark}

\numberwithin{equation}{section}

\newcommand{\1}{\mathbf{1}}
\renewcommand{\d}{\mathrm{d}}
\newcommand{\E}{\mathbb{E}}

\renewcommand{\P}{\mathbb{P}}
\newcommand{\p}{\partial}

\renewcommand{\phi}{\varphi}
\newcommand{\R}{\mathbb{R}}
\newcommand{\pp}{\mathfrak{p}}

\title[Yukawa and Brown]{Yukawa Potential, Panharmonic Measure and Brownian Motion}

\date{\today}

\author{Antti Rasila}
\address{Antti Rasila\\
Aalto University\\
School of Science\\
Department of Mathematics and Systems Analysis\\
P.O.Box 1100\\
FIN-00076 Aalto\\
FINLAND}

\author{Tommi Sottinen}
\address{Tommi Sottinen\\
University of Vaasa\\
Faculty of Technology \\
Department of Mathematics and Statistics\\
P.O.Box 700\\
FIN-65101 Vaasa\\
FINLAND}

\thanks{T. Sottinen was partially funded by the Finnish Cultural Foundation (National Foundations' Professor Pool).}



\begin{document}

\begin{abstract}
In \cite{yrs1} a Walk On Spheres (WOS) algorithm for Monte Carlo simulation of the solutions of the Yukawa and the Helmholtz PDE's was developed by using the so-called Duffin correspondence.  In this paper we investigate the foundations behind the algorithm for the case of the Yukawa PDE.  We study the panharmonic measure that is a generalization of the harmonic measure for the Yukawa PDE.  We show that there are natural stochastic definitions for the panharmonic measure in terms of the Brownian motion and that the harmonic and the panharmonic measures are all mutually equivalent. Furthermore, we calculate their Radon--Nikodym derivatives explicitly for some balls, which is a key result behind the WOS algorithm.
\end{abstract}

\maketitle

\section{Introduction and Preliminaries}

The harmonic measure is a fundamental tool in geometric function theory, and it has interesting applications in the study of bounded analytic functions, quasiconformal mappings and potential theory. For example, the harmonic measure has proven very useful in study of quasidisks and related topics (see e.g. \cite{bishob-jones,gehring-hag,krzyz}). Results involving the harmonic measure have been given by numerous authors since 1930's (see \cite{garnett-marshall} and references therein). In this paper we shall consider the panharmonic measure, which is a natural counterpart of the classical harmonic measure, where the harmonic functions related are replaced with the smooth solutions to the \emph{Yukawa equation}
\begin{equation}\label{eq:yukawa}
\Delta u(x) = \mu^2 u(x), \quad \mu^2\ge 0.
\end{equation}
The equation (\ref{eq:yukawa}) first arose from the work of the Japanese physicist Hideki Yukawa in particle physics. Here $u\colon D\to \mathbb{R}$ is a two times differentiable function and $D\subset\R^n$, $n\ge 2$, is a domain. The Yukawa equation was first studied in order to describe the nuclear potential of a point charge. This model led to the concept of the Yukawa potential (also called a screened Coulomb potential), which satisfies an equation of the type \eqref{eq:yukawa}. The Yukawa equation also arises from certain problems related to optics, see \cite{inverse}. Obviously, when $\mu=0$ we have the Laplace equation and, indeed, the results given in this paper reduce to the classical ones.

Using the terminology of Duffin \cite{duffin1,duffin2}, we call a function $u\colon D\to\R$ {\it panharmonic}, or {\it $\mu$-panharmonic}, in a domain $D$ if its second derivatives are continuous and it satisfies the Yukawa equation \eqref{eq:yukawa} for all $x\in D$. The function $u$ is called panharmonic at $x_0\in D$ if there is a neighborhood of $x_0$ where $u$ is panharmonic.

In Definition \ref{defn:panharmonic} of the panharmonic measure below, and in all that follows, we shall always assume that $n\ge 2$, although some results are true in the dimension $n=1$, also.  For Definition \ref{defn:panharmonic} we need the notions of smallness and regularity of a domain.  This is best done by using the stochastic characterization via the Brownian motion. We refer to any of the classical textbooks \cite{chung-zhao,doob,port-stone} for further details.  

Recall that the $n$-dimensional Brownian motion $W=(W(t); t\ge 0)$ starting from the point $x\in\R^n$ is the time-homogeneous Markov process with the Markov semigroup  
$$
P(t)f(x) = \E^x\left[f\big(W(t)\big)\right] 
$$
given by
$$
P(t) = e^{t\frac12\Delta},
$$
i.e., $\frac12\Delta$ is the generator of the Markov semigroup of the Brownian motion.

A domain $D\subset \R^n$ is \emph{regular} if the Brownian motion does not 
dwell on its boundary; more precisely, $D$ is (Wiener) regular if 
$$
\P^x\left[\tau_{D^c}=0\right] = 1, \quad\mbox{for all }\, x\in\p D,
$$
where $\P^x$ is the probability measure under which $\P^x[W(0)=x]=1$ and 
$$
\tau_{D} = \inf\left\{ t> 0; W(t) \in D^c\right\}
$$
is the first hitting time of the Brownian motion in the set $D^c$. We call a regular domain $D$ (Wiener) \emph{small} if a Brownian motion starting inside $D$ eventually will leave the domain, i.e., $D$ is small if 
$$
\P^x\left[\tau_D <\infty\right]=1, \quad\mbox{ for all }\, x\in D. 
$$
For example, all bounded domains are small.  Also all half-spaces are small.

The panharmonic, or $\mu$-panharmonic measure, is a generalization of the harmonic measure:
  
%%%%%%%%%%%%%%%%%%%%%%%%%%%%%%%%%%%%%%%%%%%%%%%%%%%%%%%%%%%%%%%%%%%%%%%%%%%%%%%

\begin{defn}\label{defn:panharmonic}
Let $D\subset\R^n$ be a small regular domain and let $\mu^2\ge 0$. The \emph{$\mu$-panharmonic measure} on a boundary $\p D$ with a pole at $x\in D$ is the measure $H^x_\mu(D;\cdot)$ such that \emph{any bounded} $\mu$-panharmonic function $u$ on $\bar D$ admits the representation
\begin{equation}\label{eq:ph}
u(x) = \int_{y \in \p D} u(y)\, H^x_\mu(D;\d y).
\end{equation}
\end{defn}

%%%%%%%%%%%%%%%%%%%%%%%%%%%%%%%%%%%%%%%%%%%%%%%%%%%%%%%%%%%%%%%%%%%%%%%%%%%%%%%

The existence and uniqueness of panharmonic measure will be established by Theorem \ref{thm:kakutani} and Corollary \ref{cor:yukawa-escaping} later. Indeed, by Theorem \ref{thm:kakutani} all \emph{bounded} solutions to the Dirichlet problem $\Delta u -\mu^2u = 0$ on a \emph{small} regular domain with continuous and bounded boundary data are given by the panharmonic measure as in (\ref{eq:ph}). By Corollary \ref{cor:yukawa-escaping}, if $\mu^2>0$ then the assumption that the domain is small can be removed, i.e., all bounded solutions on a regular domain are of the form (\ref{eq:ph}) if the boundary data is bounded and continuous.  Of course, it is well known that there are unbounded solutions to the Laplace equation that do not admit the harmonic measure representation.  The same is true for the Yukawa equation. We refer to Evans \cite{evans} for more details on the solutions of the Laplace equation.

Note that if we replace the `killing parameter' $\mu^2$ in the Yukawa equation \eqref{eq:yukawa} with a `creation parameter' $\lambda <0$ we obtain another important partial differential equation, the \emph{Helmholtz equation}. In principle, the stochastic approaches taken in this paper can be applied to the solutions of the Helmholtz equation if the domain $D$ is small enough compared to the parameter $\lambda$. For details, we refer to Chung and Zhao \cite{chung-zhao}. If we replace $\mu^2$ by a (positive) function, we obtain the \emph{Schr\"odinger equation}.  Again, the stochastic approaches taken in this paper can be applied, in principle, to the Schr\"odinger equation, but the results may not be mathematically very tractable. Again, we refer to Chung and Zhao \cite{chung-zhao} for details. 

The rest of the paper is organized as follows: In Section \ref{sect:yb} we show three different connections between the panharmonic measures and the Brownian motion.  The first two (Theorem \ref{thm:kakutani} and Corollary \ref{cor:yukawa-killing}) are essentially well-known.  The third one (Corollary \ref{cor:yukawa-escaping}) is new. In Section \ref{sect:equivalence} we show that the panharmonic measures and the harmonic measures are all mutually equivalent (Theorem \ref{thm:panequivalence}) and provide some corollaries, viz. we provide a domination principle for the Dirichlet problem related to the Yukawa equation (Corollary \ref{cor:dominance}) and analogs of theorems of Riesz--Riesz, Makarov and Dahlberg for the panharmonic measures (Corollary \ref{cor:panequivalence}). In section \ref{sect:average} we consider the panharmonic measures on balls and prove an analogue of the Gauss mean value theorem, or the average property, for the panharmonic functions (Theorem \ref{thm:panharmonic_average}) and as a corollary we obtain the Liouville theorem for panharmonic functions (Corollary \ref{cor:liouville}).  Finally, in Section \ref{sect:discussion} we discuss extensions to the Schr\"odinger and the Helmholtz PDE's and the Walk On Spheres simulation of PDE's.

%%%%%%%%%%%%%%%%%%%%%%%%%%%%%%%%%%%%%%%%%%%%%%%%%%%%%%%%%%%%%%%%%%%%%%%%%%%%%%%

\section{Yukawa Equation and Brownian motion} \label{sect:yb}

Let us first recall the celebrated connection between the harmonic measure and the Brownian motion first noticed by Kakutani \cite{kakutani} in the 1940's: 
the harmonic measure is the \emph{hitting measure}:
\begin{equation}\label{eq:harmonic_kakutani}
H^x(D;\d y) = \P^x\left[W(\tau_D)\in \d y, \tau_D<\infty\right].
\end{equation}

Theorem \ref{thm:kakutani} below is a variant of the Kakutani connection (\ref{eq:harmonic_kakutani}). A key ingredient in the variant is the following disintegration of the harmonic measure on the time the associated Brownian motion hits the boundary $\p D$:

\begin{lemma}\label{lemma:harmonic_kernel}
Let $D\subset \R^n$ be a regular domain and $x\in D$. Then 
$$
H^x(D;\d y) = \int_{t=0}^{\infty} h^x(D;\d y, t)\, \d t,
$$
where 
\begin{equation}\label{eq:harmonic_kernel}
h^x(D;\d y,t) = \P^x\left[W(\tau_D)\in \d y\mid \tau_D=t\right]\frac{\d\P^x}{\d t}\left[ \tau_D\le t\right]
\end{equation}
is the \emph{harmonic kernel}.
\end{lemma}

\begin{proof}
First, we show the existence of the regular conditional distribution
\begin{equation}\label{eq:regular_law}
\pp^x(\d y\,|\,t) = \P^x\left[W(\tau_D)\in \d y\mid \tau_D=t\right].
\end{equation}
For this, we note that the random vector $(W(\tau_D),\tau_D)$ can be considered as a function from a space of continuous functions that are the Brownian trajectories equipped with the metric
$$
d(f,g) = \sum_{T=1}^\infty 2^{-T}\Big\|f\1_{[T-1,T)}-g\1_{[T-1,T)}\Big\|_\infty.
$$
For Brownian trajectories the metric $d$ is almost surely finite due to the independent increments of the Brownian motion and the Borel--Cantelli lemma.
Also, with the metric $d$, the space of Brownian paths is a Polish space. 
Now, by Theorem A1.2 of \cite{kallenberg} Polish spaces are Borel spaces.  Consequently, for any fixed $x\in D$, by Theorems 6.3 and 6.4 of \cite{kallenberg}, the probability kernel \eqref{eq:regular_law} exist and is measurable with respect to $t$.  Consequently, the harmonic kernel is measurable with respect to $t$.

%Second, we show that the harmonic kernel is measurable w.r.t. to $t$.  To this end, it is enough to show that the kernel $\pp^x(\d y\,|\,\cdot)$ defined by \eqref{eq:regular_law} is measurable.  But  
%$$
%\pp^x(\d y \,|\, t) = \frac{
%\P^x\left[ W(\tau_D) \in \d y\,,\, \tau_D\in \d t\right]}{\P^x[\tau_D\in\d t]}
%$$
%is a (Radon--Nikodym) derivative of two (sub-)probability measures.  Hence it is measurable.

Second, we show that the distribution of the hitting time $\tau_D$ is absolutely continuous with respect to the Lebesgue measure.  Let $\varepsilon>0$ be small enough so that $B=B(x,\varepsilon)\subset D$. Then $\tau_D = \tau_B + (\tau_D-\tau_B)$.  Now, the distribution of $\tau_B$ is absolutely continuous; see, e.g., the section of Bessel processes in Borodin and Salminen \cite{borodin-salminen}. Also, due to the rotation symmetry of the  Brownian motion, $\tau_B$ and $\tau_D-\tau_B$ are independent.  Hence, by disintegration and independence, we obtain that
\begin{eqnarray*}
\P^x[\tau_D \in \d t] &=& \P^x[\tau_{B} + (\tau_{D}-\tau_B) \in \d t] \\
&=& \int_{s=0}^\infty \P^x[t+(\tau_D-\tau_B) \in \d s\,|\, \tau_B=t]\,\P^x[\tau_B\in \d t] \\
&=&
\int_{s=0}^\infty \P^x[t+(\tau_D-\tau_B)\in \d s]\,\P^x[\tau_B\in \d t]\\
&=& \varphi^x(t)\,\P^x[\tau_B\in \d t]. 
\end{eqnarray*}
Thus, the distribution of $\tau_D$ is absolutely continuous, when the distribution of $\tau_B$ is absolutely continuous.

Third, we show that the formula \eqref{eq:harmonic_kernel} holds. By disintegrating and conditioning, and by using the continuity of the distribution of $\tau_D$, we obtain that
\begin{eqnarray*}
\lefteqn{\P^x\left[W(\tau_D)\in \d y, \tau_D<\infty\right] }\\
&=&
\int_{t=0}^{\infty} \P^x\left[W(\tau_D)\in \d y, \tau_D\in \d t\right] \\
&=&
\int_{t=0}^{\infty} \P^x\left[W(\tau_D)\in \d y\mid \tau_D=t\right]\P^x\left[ \tau_D\in \d t\right] \\
&=&
\int_{t=0}^{\infty} \P^x\left[W(\tau_D)\in \d y\mid \tau_D=t\right]\frac{\d\P^x}{\d t}\left[ \tau_D\le t\right]\, \d t.
\end{eqnarray*}
The claim follows now from the Kakutani connection (\ref{eq:harmonic_kakutani}).
\end{proof}

%% Kakutani for Yukawa %%%%%%%%%%%%%%%%%%%%%%%%%%%%%%%%%%%%%%%%%%%%%%%%%%%%%%%%

The following theorem \ref{thm:kakutani} is a version of the Kakutani theorem \cite{kakutani} for the Yukawa equation.  In some sense it is a special case of the Kakutani connection to the Schr\"odinger equation studied extensively by Chung and Zhao \cite{chung-zhao}.  However, it seems that this version with unbounded and non-small domain $D$ does not appear in any classical texts. 

\begin{thm}\label{thm:kakutani}
Let $D\subset\R^n$ be a regular domain and let $f:\p D \to \R$ be bounded and continuous. 
\begin{enumerate}
\item
Then
\begin{equation}\label{eq:discounting}
u(x) = \E^x\left[e^{-\frac{\mu^2}{2}\tau_D}f(W(\tau_D))\,;\, \tau_D<\infty\right] 
\end{equation}
is \emph{a} solution to the Yukawa--Dirichlet problem 
$$
\left\{
\begin{array}{rclll}
\Delta u &=& \mu^2u &\mbox{ on }& D, \\
u        &=& f      &\mbox{ on }& \p D. 
\end{array}
\right.
$$
\item
Moreover, if $u$ is bounded and $D$ is small then (\ref{eq:discounting}) is \emph{the only} solution to the Yukawa--Dirichlet problem.
\item
As a consequence, the harmonic measure admits the representation
\begin{equation}\label{eq:rep1}
H_\mu^x(D;\d y) = \int_{t=0}^{\infty} e^{-\frac{\mu^2}{2}t}\, h^x(D;\d y,t)\, \d t,
\end{equation}
where $h^x(D;\cdot,\cdot)$ is the harmonic kernel defined in (\ref{eq:harmonic_kernel}).
\end{enumerate}
\end{thm}

%%%%%%%%%%%%%%%%%%%%%%%%%%%%%%%%%%%%%%%%%%%%%%%%%%%%%%%%%%%%%%%%%%%%%%%%%%%%%%%

\begin{proof}
The first and the second claim of Theorem \ref{thm:kakutani} follow from the classical Kakutani theorem, cf., e.g., \cite{durrett} sections 4.4. and 4.6.  Indeed, note that the difficulties involving the Schr\"odinger in \cite{durrett} Section 4.6. vanish, since
$$
\E^x\left[e^{-\frac{\mu^2}{2}\tau_D}\right] \le 1.
$$ 

To show the third claim, we condition on $\{\tau_D=t\}$ and use the law of total probability:
\begin{eqnarray*}
u(x) &=& \E^x\left[e^{-\frac{\mu^2}{2}\tau_D}f(W(\tau_D))\,;\, \tau_{D}<\infty\right] \\
%&=& \int_{t=0}^{\infty} \E^x\left[e^{-\frac{\mu^2}{2}t}f(W(t))\,\Big|\, \tau_D=t\right]\,\P^x\left[\tau_D\in\d t\right] \\
%&=& \int_{t=0}^{\infty} e^{-\frac{\mu^2}{2}t}\, \E^x\left[f(W(t))\,\Big|\, \tau_D=t\right]\,\P^x\left[\tau_D\in\d t\right] \\
%&=& \int_{t=0}^{\infty} e^{-\frac{\mu^2}{2}t} \int_{y\in\R^n} f(y)\,\P^x\left[ W(t)\in \d y\,\big|\, \tau_D=t\right]\P^x\left[\tau_D\in\d t\right]  \\
%&=& \int_{y\in\p D} f(y)\, \int_{t=0}^{\infty} e^{-\frac{\mu^2}{2}t}\, \P^x\left[ W(t)\in \d y\,\big|\, \tau_D=t\right]\P^x\left[\tau_D\in\d t\right]  \\
%&=& \int_{y\in\p D} f(y)\, \int_{t=0}^{\infty} e^{-\frac{\mu^2}{2}t}\, \P^x\left[ W(t)\in \d y\,,\,\tau_D\in\d t\right]  \\
&=& \int_{y\in\p D} f(y)\, \int_{t=0}^{\infty} e^{-\frac{\mu^2}{2}t}\, h^x(D;\d y, t)\, \d t  \\
&=& \int_{y\in \p D} f(y)\, H_\mu^x(D;\d y).
\end{eqnarray*}
\end{proof}

\begin{rem}\label{rem:c3boundary}
Unfortunate, even for very simple $D$ the harmonic kernel \eqref{eq:harmonic_kernel} is quite difficult to find out.  The same is true for the regular conditional distribution \eqref{eq:regular_law}. 
For smooth boundaries $\p D$ one can try the following approach: If $\p D$ is smooth, then the harmonic kernel $h^x(D;\d y ,t)$ is absolutely continuous with respect to the Lebesgue measure $\d y$. Indeed, define $p:\R_+\times\R^n\to \R_+$ by
\begin{equation}\label{eq:btk}
p(t,x) = \frac{1}{(2\pi t)^{n/2}}\exp\left(-\frac{\|x\|^2}{2t}\right).
\end{equation}
Then $p$ is the \emph{Brownian transition kernel}:
$$
p(t,x-y)\, \d y = \P^x\left[ W(t) \in \d y\right] 
$$
and, due to \cite[Theorem 1]{hsu} the harmonic kernel can be written as
$$
h^x(D; \d y,t) = \frac12 \frac{\p p}{\p\mathrm{n}_y}(D;t,x-y)\, \d y, 
$$ 
where $\mathrm{n}_y$ is the inward normal at $y\in\p D$ and $p(D;\cdot,\cdot)$ is the transition density of a Brownian motion that is killed when it hits the boundary $\p D$, which can be written as
\begin{equation}\label{eq:killed-transition}
p(D;t,x-y) 
= p(t,x-y) - \E^x\Big[p\big(t-\tau_D,W(\tau_D)-y\big)\,;\, \tau_D<t\Big] 
\end{equation}
due to \cite[formula (3) on page 34]{port-stone}.

Consequently, for $C^3$ boundaries the harmonic measure admits a \emph{Poisson kernel} representation and therefore, due to the representation (\ref{eq:rep1}) the panharmonic measure also admits a Poisson kernel representation:
\begin{eqnarray*}
H^x_\mu(D;\d y) &=& \int_{t=0}^\infty e^{-\frac{\mu^2}{2}t} h^x(D;\d y, t)\, \d t \\
&=& \int_{t=0}^\infty e^{-\frac{\mu^2}{2}t} \frac{1}{2}\frac{\p p}{\p \mathrm{n}_y}(D;t,x-y)\, \d y\, \d t \\
&=&
\left[\frac12 \int_{t=0}^\infty e^{-\frac{\mu^2}{2}t}\frac{\p p}{\p \mathrm{n}_y}(D;t,x-y)\, \d t\right]\, \d y.
\end{eqnarray*}
\end{rem}

Theorem \ref{thm:kakutani} gives an interpretation of the panharmonic measure in terms of exponentially discounted Brownian motion.  Let us give a second interpretation in terms of exponentially killed Brownian motion. Indeed, exponential discounting is closely related to exponential killing. The \emph{exponentially killed Brownian motion} $W_\mu$ is 
$$
W_\mu(t) = W(t)\1_{\{Y_\mu>t\}} + \dagger\1_{\{Y_\mu\le t\}},
$$
where $\dagger$ is a \emph{coffin state}\footnote{By convention $f(\dagger)=0$ for all functions $f$.} and $Y_\mu$ is an independent exponential random variable with mean $2/\mu^2$, i.e. 
$
\P\left[Y_\mu > t\right] = e^{-\frac{\mu^2}{2}t}.
$
Let
$$
\tau_D^\mu = \inf\left\{ t>0\,;\, W_\mu(t) \in D^c\right\}.
$$
Then we have the following representation of the panharmonic measure:

\begin{cor}\label{cor:yukawa-killing}
Let $D\subset\R^n$ be a regular domain. Then the panharmonic measure admits the representation
\begin{equation}\label{eq:yukawa-killing}
H^x_\mu(D;\d y) = \P^x\left[ W_\mu(\tau_D^\mu) \in\d y\,;\, \tau_D^\mu < \infty\right].
\end{equation}
\end{cor}

\begin{proof}
Let $f:\p D\to\R$ be bounded. Then, by Theorem \ref{thm:kakutani} and the independence of $W$ and $Y_\mu$,
\begin{eqnarray*}
\lefteqn{\int_{y\in \p D} f(y)\, H_\mu^x(D;\d y)} \\
&=& \E^x\left[ e^{-\frac{\mu^2}{2}\tau_D} f\left(W(\tau_D)\right); \tau_D<\infty\right] \\
&=& \int_{y\in\p D} f(y)\int_{t=0}^{\infty} e^{-\frac{\mu^2}{2}t}\,\P^x\left[W(t)\in \d y, \tau_D\in \d t \right] \\
&=& \int_{y\in\p D} f(y)\int_{t=0}^{\infty} \P^x\left[Y_\mu > t\right]\P^x\left[W(t)\in \d y, \tau_D\in \d t \right] \\ 
&=& \int_{y\in\p D} f(y)\int_{t=0}^{\infty} \P^x\left[Y_\mu > t,W(t)\in \d y, \tau_D\in \d t \right] \\
&=& \int_{y\in\p D} f(y)\int_{t=0}^{\infty} \P^x\left[W_\mu(t)\in \d y, \tau_D^\mu\in \d t \right] \\
&=& \E^x\left[f\left(W_\mu(\tau_D^\mu)\right); \tau_D^\mu < \infty\right]. 
\end{eqnarray*}
Since $f$ was arbitrary, the claim follows.
\end{proof}

The two representations, Theorem \ref{thm:kakutani} and Corollary \ref{cor:yukawa-killing}, for the panharmonic measures are, at least in spirit, classical.  
Now we give a third representation for the panharmonic measure in terms of an \emph{escaping Brownian motion}.
This representation is apparently new in spirit.  The representation is due to the following \emph{Duffin correspondence} \cite{duffin1}: Let $D\subset\R^n$ be a regular domain and let $u:D\to\R$. Let $I\subset\R$ be any open interval that contains $0$. Set $\bar D=D\times I$ and define $\bar u:\bar D\to\R$ by
\begin{equation}\label{eq:harmonic_duffin}
\bar u(\bar x) = \bar u(x,\tilde x) = u(x)\cos(\mu \tilde x).
\end{equation}

\begin{thm}\label{thm:duffin-correspondence}
The function $\bar u$ defined by (\ref{eq:harmonic_duffin}) is harmonic on $\bar D$ if and only if $u$ is $\mu$-panharmonic on $D$. 
\end{thm}

\begin{proof}
Let us first show that $D$ is regular if and only if $\bar D$ is regular. Let $\bar W= (W,\tilde W)$ be $(n+1)$-dimensional Brownian motion. Denote
\begin{eqnarray*}
\tau &=& \inf\{t>0\,;\, W(t) \in D^c\}, \\
\tilde\tau &=& \inf\{t>0\,;\, \tilde W(t) \in I^c\}, \\
\bar\tau &=& \inf\{t>0\,;\, \bar W(t) \in \bar D^c\}.
\end{eqnarray*}
Note that for  $\{\tilde\tau=\tilde x\}$ to happen, $\tilde x$ must be an endpoint of the interval $I$. Then, by independence of $W$ and $\tilde W$,
\begin{eqnarray*}
\P^{x,\tilde x}[\bar \tau = 0] &=& \P^{x,\tilde x}[\tau = 0, \tilde\tau =0] \\
&=& \P^x[\tau = 0]\P^{\tilde x}[\tilde \tau =0] \\
&=& \P^x[\tau=0],
\end{eqnarray*}
since $I$ is obviously regular. This shows that $\bar D$ is regular if and only if $D$ is regular.

Let us then show that $u$ satisfies the Laplace equation if and only if $\bar u$ satisfies the Yukawa equation.  But this is straightforward calculus: 
\begin{eqnarray*}
\Delta_{\bar x} \bar u(\bar x) 
&=& \Delta_{x,\tilde x} \left[u(x)\cos(\mu\tilde x)\right] \\
&=& \cos(\mu\tilde x)\Delta_x u(x) + u(x)\frac{\d^2}{\d\tilde x^2}\cos(\mu\tilde x)\\
&=& \cos(\mu\tilde x)\Delta_x u(x) - \mu^2\cos(\mu\tilde x) \\
&=& \cos(\mu\tilde x)\left(\Delta_x u(x) - \mu^2u(x)\right) \\
&=& 0
\end{eqnarray*}
if and only if $\Delta_x u(x) = \mu^2 u(x)$.
\end{proof}

Let $\tilde W$ be a $1$-dimensional standard Brownian motion that is independent of $W$. Then $\bar W = (W,\tilde W)$ is a $(n+1)$-dimensional standard Brownian motion. 

Now the idea how to use the Duffin correspondence is clear.  We start the Brownian particle $\bar W$ and count the boundary data on the side of the cylinder $\bar D=D\times I$, if the Brownian motion does not escape the cylinder from the bottom or from the top.  In that case we count zero in the boundary.  Whence the name \emph{escaping Brownian motion}. 

\begin{cor}\label{cor:yukawa-escaping}
Let $D\subset\R^n$ be a regular domain. Then the panharmonic measure admits the representation
\begin{eqnarray}\label{eq:yukawa-escaping}
\lefteqn{H^x_\mu(D;\d y)} \nonumber \\ \nonumber
 &=& \E^{x,0}\left[\cos\left(\mu \tilde W(\tau_D)\right); W(\tau_D)\in \d y, \sup_{t\le \tau_D}| \tilde W(t)|<\frac{\pi}{2\mu}\right] \\
&=& \label{eq:rep2}
\int_{\tilde y=-\frac{\pi}{2\mu}}^{\frac{\pi}{2\mu}} \cos\left(\mu \tilde y\right) H^{x,0}\left(D\times\left(-\frac{\pi}{2\mu},\frac{\pi}{2\mu}\right)\,;\, \d y\otimes \d\tilde y\right).
\end{eqnarray}
Here we have chosen $I=(-\frac{\pi}{2\mu},\frac{\pi}{2\mu})$ in the Duffin correspondence.

Consequently, all bounded solutions to the Yukawa--Dirichlet problem on a regular domain with $\mu^2>0$ and continuous and bounded boundary data are given by the panharmonic measure.
\end{cor}

\begin{proof}
The claim follows by combining the Kakutani connection (\ref{eq:harmonic_kakutani}) with the Duffin correspondence (\ref{eq:harmonic_duffin}) by noticing that it is enough to integrate over $\p D\times (-\pi/(2\mu),\pi/(2\mu))$ since $\cos(\mu\tilde y)=0$ on the boundary $\p(-\pi/(2\mu),\pi/(2\mu))$. 

Finally, note that for regular domain $D$, the domain $\bar D$ is regular and small.
\end{proof}

\begin{rem}
Representation (\ref{eq:rep2}) is exceptionally well-suited for calculations of the panharmonic measures on upper half-spaces $\mathbb{H}_+^n=\{x\in \R^{n}; x_n > 0 \}$. Indeed, Duffin \cite[Theorem 5]{duffin1} used it to calculate the Poisson kernel representation for panharmonic measures in the dimension $n=2$.  Similar calculations can be carried out for the general case $n\ge 2$, also. 
\end{rem}

%%%%%%%%%%%%%%%%%%%%%%%%%%%%%%%%%%%%%%%%%%%%%%%%%%%%%%%%%%%%%%%%%%%%%%%%%%%%%%%

\section{Equivalence of Harmonic and Panharmonic Measures}
\label{sect:equivalence}

The probabilistic interpretation provided by Corollary \ref{cor:yukawa-killing} implies that the harmonic measure and the panharmonic ones are equivalent.  Indeed, the harmonic measure counts the Brownian particles on the boundary and the panharmonic measures count the killed Brownian particles on the boundary.  But the killing happens with independent exponential random variables.  So, if the Brownian motion can reach the boundary with positive probability, so can the killed Brownian motion; and vice versa. Also, it does not matter, as far as the equivalence is concerned, what is the starting point of the Brownian motion, killed or not. 

Theorem \ref{thm:panequivalence} below makes the heuristics above precise. As corollaries of Theorem \ref{thm:panequivalence} we obtain a domination principle for the Dirichlet problem related to the Yukawa equation (Corollary \ref{cor:dominance}) and analogs of theorems of Riesz--Riesz, Makarov and Dahlberg for the panharmonic measures (Corollary \ref{cor:panequivalence}).

The same arguments that give the existence of the regular conditional law \eqref{eq:regular_law} in the proof of Lemma \ref{lemma:harmonic_kernel} also give the existence and regular measurability of the following conditional Radon--Nikodym derivative 
\begin{equation}\label{eq:radon-nikodym}
Z_\mu^x(D;y) = \E^x\left[e^{-\frac{\mu^2}{2}\tau_D}\,\Big|\, W(\tau_D)=y\right].
\end{equation}

\begin{thm}\label{thm:panequivalence}
Let $D$ be a regular domain.  Then all the panharmonic measures $H_\mu^x(D;\cdot)$, $\mu\ge0, x\in D$, are mutually equivalent. The Radon-Nikodym derivative 
of $H_\mu^x(D;\cdot)$ with respect to $H^x(D;\cdot)$ is the function $Z_\mu^x(D;\cdot)$ given by (\ref{eq:radon-nikodym}).
Moreover $Z_\mu^x(D;y)$ is strictly decreasing in $\mu$, and  $0<Z_\mu^x(D;y)\le 1$. 
\end{thm}

\begin{rem}
By Corollary \ref{cor:yukawa-killing} the Radon--Nikodym derivative $Z_\mu^x(D;\cdot)$ in (\ref{eq:radon-nikodym}) can be interpreted as the probability that a Brownian motion killed with intensity $\mu^2/2$, that would exit the domain $D$ at $y\in \p D$, survives to the boundary $\p D$:
\begin{equation}\label{eq:radon-nikodym2}
Z_\mu^x(D;y) = \P^x\left[Y_\mu>\tau_D \,|\, W(\tau_D)=y\right],
\end{equation}
where $Y_\mu$ is exponentially distributed random variable with mean $2/\mu^2$ that is independent of the Brownian motion $W$.
\end{rem}

\begin{proof}[Proof of Theorem \ref{thm:panequivalence}]
Let $x,y\in D$ and let $D_{0}\subset D$ be a subdomain of $D$ such that $x\in D_{0}$ and $y\in\p D_{0}$. Then, the Markov property of the Brownian motion and the Kakutani connection (\ref{eq:harmonic_kakutani}), we have
$$
H^x(D;A) = \int_{y\in \p D_{0}} H^{y}(D;A)H^{x}(D_{0};\d y)
$$
for all measurable $A\subset\p D$. This shows the harmonic measures $H^x(D;\cdot)$, $x\in D$, are mutually equivalent. 

To see that $Z_\mu^x(D;\cdot)$ is the Radon--Nikodym derivative, note that, by the representation (\ref{eq:rep1}) and the Kakutani connection (\ref{eq:harmonic_kakutani}), 
\begin{eqnarray*}
H_\mu^x(D;\d y) &=& 
\int_{t=0}^\infty e^{-\frac{\mu^2}{2}t}\, h^x(D;\d y,t)\,\d t \\
&=&
\int_{t=0}^\infty e^{-\frac{\mu^2}{2}t}\,\P^x\left[ W(\tau_D)\in \d y, \tau_D\in \d t\right] \\
&=& \int_{y\in \p D}\E^x\left[e^{-\frac{\mu^2}{2}\tau_D}\,\Big|\, W(\tau_D)=y\right]\,\P^x\left[W(\tau_D)\in\d y\right]\\
&=& \int_{y\in \p D} Z_\mu^x(D;y)\, H^x(D;\d y).  
\end{eqnarray*}

Finally,  the fact that $0<Z_\mu^x(D;\cdot)\le 1$ is obvious from the representation (\ref{eq:radon-nikodym}). The fact that $Z_\mu^z(D;\cdot)$ is strictly decreasing follows immediately from the representation \eqref{eq:radon-nikodym2}.  
\end{proof}

From Theorem \ref{thm:panequivalence} we obtain immediately the following \emph{domination principle} for the Dirichlet problem related to panharmonic functions:

\begin{cor}\label{cor:dominance}
Let $D$ be a regular domain and let $u_\mu$ by $\mu$-panharmonic and $u_\nu$ be $\nu$-panharmonic, respectively, on $D$ with $\nu\le\mu$. Then,
$u_\nu \le u_\mu$ on $\p D$ implies $u_\nu \le u_\mu$ on $D$.
\end{cor}  

Since domains with rectifiable boundary are regular, we obtain immediately from Theorem \ref{thm:panequivalence} the following analogs of the theorems of F. Riesz and M. Riesz, Makarov and Dalhberg (see \cite{riesz-riesz}, \cite{makarov} and \cite{dahlberg}, respectively).

\begin{cor}\label{cor:panequivalence}
Let $\mathcal{H}^s(D;\cdot)$ be the $s$-dimensional Hausdorff measure on $\p D$.
\begin{enumerate}
\item 
Let $D\subset\mathbb{R}^2$ be a simply connected planar domain bounded by a rectifiable curve. Then $H_\mu^x(D;\cdot)$ and $\mathcal{H}^1(D;\cdot)$ are equivalent for all $\mu\ge0$ and $x\in D$.
\item
Let $D\subset\mathbb{R}^2$ be a simply connected planar domain. If $E\subset\partial D$ and $\mathcal{H}^s(D;E)=0$ for some $s<1$, then $H_\mu^x(D;E)=0$ for all $\mu\ge0$ and $x\in D$. Moreover, $H_\mu^x(D;\cdot)$ and $\mathcal{H}^t(D;\cdot)$ are singular for all $\mu\ge 0$ and $x\in D$ if $t>1$.
\item
Let $D\subset\mathbb{R}^n$ is a bounded Lipschitz domain. Then $H_\mu^x(D;\cdot)$ and $\mathcal{H}^{n-1}(D;\cdot)$ are equivalent for all $\mu\ge 0$ and $x\in D$.
\end{enumerate}
\end{cor}

%%%%%%%%%%%%%%%%%%%%%%%%%%%%%%%%%%%%%%%%%%%%%%%%%%%%%%%%%%%%%%%%%%%%%%%%%%%%%%%

\section{The Average Property for Panharmonic Measures and Functions} \label{sect:average}

By using the representation (\ref{eq:rep1}) one can calculate the panharmonic measures if one can calculate the corresponding harmonic kernels. Or, equivalently, one can calculate the panharmonic measures if one can calculate the corresponding harmonic measures and the Radon--Nikodym derivatives given by (\ref{eq:radon-nikodym}). 

The harmonic kernels for balls are calculated in \cite{hsu}. We do not, however, present the general formula here. Instead, we confine ourselves in the case where the center of the ball and the pole of the panharmonic measure coincide, and give the Gauss mean value theorem, or the average property, for panharmonic measures. As a corollary we have the Liouville theorem for the panharmonic measures.

Let $D\subset\R^n$ be a regular domain.  For the harmonic measure the \emph{Gauss mean value theorem} states that a function $u:D\to\R$ is harmonic if and only if for all balls $B_n(x,r)\subset D$ we have the \emph{average property}
$$
u(x) = \int_{y\in \p B_n(x,r)} u(y)\, \sigma_n(r;\d y),
$$
where
$$
\sigma_n(r; \d y) = \frac{\Gamma(n/2)}{2\pi^{n/2}} r^{1-n}\, \d y
$$
is the uniform probability measure on the sphere $\p B_n(x,r)$.

For the panharmonic measures the situation is similar to the harmonic measure: the only difference is that the uniform probability measure has to be replaced by a uniform sub-probability measure that depends on the killing parameter $\mu$ and the radius of the ball $r$. Indeed, denote 
\begin{equation}\label{eq:psi-def}
\psi_n(\mu ) = 
\frac{\mu^\nu}{2^\nu\Gamma(\nu+1)I_\nu(\mu)}, \quad \mu>0, 
\end{equation}
where $\nu=(n-2)/2$ and
$$
I_\nu(x) = \sum_{m=0}^\infty \frac{1}{m!\Gamma(m+\nu+1)}\left(\frac{x}{2}\right)^{2m+\nu}
$$
is the modified Bessel function of the first kind of order $\nu$.

\begin{thm}\label{thm:panharmonic_average}
Let $D\subset\R^n$ be a regular domain and let $\mu>0$. A function $u:D\to\R$ is $\mu$-panharmonic if and only if it has the average property:
\begin{equation*}\label{eq:panharmonic_average}
u(x) = \psi_n(\mu r)\int_{y\in \p B_n(x,r)} u(y)\, \, \sigma_n(r;\d y).
\end{equation*}
for all open balls $B_n(x,r)\subset D$.  Equivalently,
\begin{equation*}\label{eq:panharmonic_measure_average}
H^x_\mu\left(B_n(x,r); \d y\right) = \psi_n(\mu r)\, \sigma_n(r;\d y).
\end{equation*}
\end{thm}

\begin{rem}
Theorem \ref{thm:panharmonic_average} states that $\psi_n(\mu r)$ is the Radon--Nikodym derivative:
$$
\psi_n(\mu r) = Z_\mu^x\left(B_n(x,r); y\right) 
= \E^x\left[e^{-\frac{\mu^2}{2}\tau_{B_n(x,r)}}\Big| W\left(\tau_{B_n(x,r)}\right)=y\right].
$$
\end{rem}

\begin{proof}[Proof of Theorem \ref{thm:panharmonic_average}]
Note that we may assume that $x=0$.

Denote by $\tau_r^n$ the first hitting time of the Brownian motion $W$ on the boundary $\p B_n(0,r)$. I.e., $\tau_r^n$ is identical in law with the first hitting time of the Bessel process with index $\nu=(n-2)/2$ reaches the level $r$ when it starts from zero.

From the rotation symmetry of the Brownian motion it follows that the hitting place is uniformly distributed on $\p B_n(0,r)$ for all hitting times $t$. Consequently, by Theorem \ref{thm:kakutani} and the independence of the hitting time $\tau_r^n$ and place $W(\tau_r^n)$
\begin{eqnarray*}
H_\mu^0\left(B_n(0,r); \d y\right) 
&=& \E^0\left[e^{-\frac{\mu^2}{2}\tau_r^n}; W(\tau_r^n)\in \d y\right] \\
&=& \E^0\left[e^{-\frac{\mu^2}{2}\tau_r^n}\right] \P^0\left[W(\tau_r^n)\in \d y\right] \\
&=& \E^0\left[e^{-\frac{\mu^2}{2}\tau_r^n}\right] \sigma_n(r;\d y).
\end{eqnarray*}
The hitting time distributions for the Bessel process are well-known.  By, e.g., Wendel \cite[Theorem 4]{wendel},
$$
\E^0\left[e^{-\frac{\mu^2}{2}\tau_r^n}\right] =
\frac{(\mu r)^\nu}{2^\nu\Gamma(\nu+1)I_\nu(\mu r)}.
$$
The claim follows from this.
\end{proof}

\begin{rem}\label{rem:psifun}
The Radon--Nikodym derivative, or the `killing constant', $\psi_n(\mu)$ is rather complicated.  However, some of its properties are easy to see:
\begin{enumerate}
\item $\psi_n(\mu)$ is continuous in $\mu$,
\item $\psi_n(\mu)$ is strictly decreasing in $\mu$,
\item $\psi_n(\mu) \to 0$ as $\mu\to\infty$,
\item $\psi_n(\mu) \to 1$ as $\mu\to 0$,
\item $\psi_n(\mu)$ is increasing in $n$. 
\end{enumerate} 
The items (i)--(iv) are clear since $\psi_n(\mu)$ is the probability that an exponentially killed Brownian motion started from the origin with killing intensity $\mu^2/2$ is not killed before it hits the boundary of the unit ball. A non-probabilistic argument for (i)--(iv) is to note that
$$
\psi_n(\mu r) = \E^0\left[e^{-\frac{\mu^2}{2}\tau_r^n}\right].
$$
and use the monotone convergence.
The item (v) is somewhat surprising: the higher the dimension $n$, the more likely it is for the killed Brownian motion to survive to the boundary of the unit ball. A possible intuitive explanation is that the higher the dimension the more transitive the unit ball is combined with the remarkable result by Ciesielski and Taylor \cite{ciesielski-taylor} that probability distribution for the total time spent in a ball by $(n+2)$-dimensional Brownian motion is the same as the probability distribution of the hitting time of $n$-dimensional Brownian motion on the boundary of the ball.      
\end{rem}

\bigskip
\begin{center}
\includegraphics[width=\textwidth]{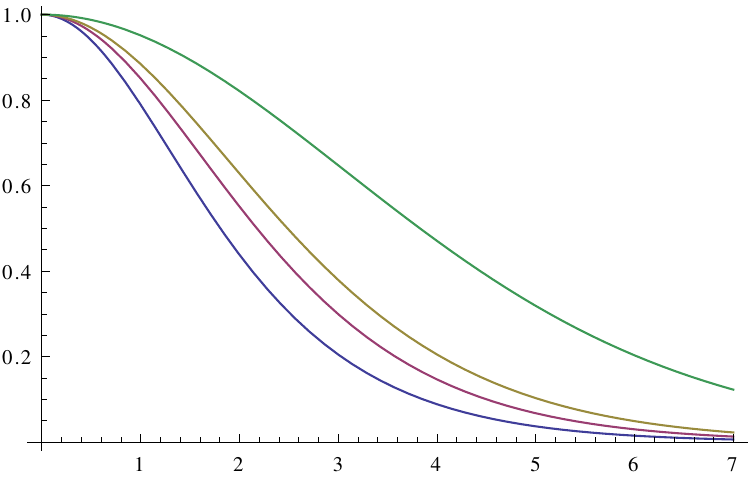}
Function $\psi_n$ with (from bottom to top) $n=2,3,4,10$.
\end{center}
\bigskip

\begin{cor}\label{cor:liouville}
Let $u$ be panharmonic on the entire space $\R^n$. If $u$ is bounded, then $u$ is constant. 
\end{cor}

\begin{proof}
By Theorem \ref{thm:panharmonic_average}
\begin{eqnarray*}
\lefteqn{\left|u(x)-u(0)\right|} \\
&=&
\left|\psi_n(\mu r)\int_{\p B_n(x,r)} u(y)\, \, \sigma_n(r;\d y)
- \psi_n(\mu r)\int_{\p B_n(0,r)} u(y) \, \, \sigma_n(r;\d y) 
\right| \\
&\le&
\left|\psi_n(\mu r)\int_{\p B_n(x,r)} u(y)\, \, \sigma_n(r;\d y)\right| +
\left|\psi_n(\mu r)\int_{\p B_n(0,r)} u(y)\, \, \sigma_n(r;\d y)\right| \\
&\le&
2\psi_n(\mu r){\| u\|}_{\infty},
\end{eqnarray*}
which tends to $0$ as $r\to\infty$ by property (iii) or Remark \ref{rem:psifun}.
\end{proof}

%%%%%%%%%%%%%%%%%%%%%%%%%%%%%%%%%%%%%%%%%%%%%%%%%%%%%%%%%%%%%%%%%%%%%%%%%%%%%%%

\section{Discussion on Extensions and Simulation} \label{sect:discussion}

The Yukawa equation (\ref{eq:yukawa}) is a special case of the Schr\"odinger equation
\begin{equation}\label{eq:schrodinger}
\Delta u(x) = q(x)u(x).
\end{equation}
The Schr\"odinger equation and its connection to the Brownian motion has been studied e.g. by Chung and Zhao \cite{chung-zhao}.  Our investigation here can be seen as a special case. For example, analogs of Theorem \ref{thm:kakutani} and Corollary \ref{cor:yukawa-killing} are known for the Schr\"odinger equation. However, analogs of the Duffin correspondence (\ref{eq:harmonic_duffin}) and Corollary \ref{cor:yukawa-escaping} are not known even to exist. Moreover, the results given here cannot easily be calculated for the Schr\"odinger equation. The problem is that the prospective Radon--Nikodym derivate of the measure associated with the solutions of the Schr\"odinger equation with respect to the harmonic measure takes the form
\begin{equation}\label{eq:radon-nikodym-schrodinger}
Z_q^x(D;y) = \E^x\left[ e_q(\tau_D) \big| W(\tau_D) =y \right],
\end{equation}
where 
$$
e_q(t) = e^{-\frac12\int_0^t q(W(s))\,\d s}
$$
is the so-called \emph{Feynman--Kac functional}.  Thus, we see that in order to calculate the Radon--Nikodym derivative we need to know the joint density of the Feynman--Kac functional and the Brownian motion when the Brownian motion hits the boundary $\p D$. If $q$ is constant, i.e., we have either the Yukawa equation or the Helmholtz equation, then it is enough to know the joint distribution of the hitting time and place of the Brownian motion on the boundary $\p D$.  These distributions are well-studied, see e.g. \cite{borodin-salminen,ciesielski-taylor,hsu,kent,levy}, but few joint distributions involving the Feynman--Kac functionals are known.

%It would be interesting to calculate the Radon--Nikodym derivative \eqref{eq:radon-nikodym-schrodinger} for, say, balls and half-spaces, and thus reproduce the related results of this paper to the Schr\"odinger equation.  

In addition to the Yukawa equation, the other important special case of the Schr\"odinger equation (\ref{eq:schrodinger}) is the Helmholtz equation,
\begin{equation}\label{eq:helmholtz}
\Delta u(x) = -\lambda u(x), \quad \lambda\ge 0.
\end{equation}
It is possible to provide a Duffin correspondence for the Helmholtz equation also.  Indeed, e.g., setting
$$
\bar u(\bar x) = \bar u(x,\tilde x) = u(x)\cosh(\lambda \tilde x)
$$
provides a correspondence (see \cite{yrs1} for details).
Thus, our results extend in a straightforward manner to the Helmholtz equation \eqref{eq:helmholtz} for domains that are small enough with respect to the creation parameter $\lambda$ so that the associated Feynman--Kac functional is finite:
\begin{equation}\label{eq:gauge}
\E^x\left[e^{\frac{\lambda}{2}\tau_D}\right] < \infty.
\end{equation}

Finally, we note Theorem \ref{thm:kakutani}, Corollary \ref{cor:yukawa-killing} and Corollary \ref{cor:yukawa-escaping} give three different ways to simulate the panharmonic measures.  Indeed, in \cite{yrs1,yrs2} the classical Walk On Spheres algorithm due to Muller \cite{muller} was extended for the Yukawa PDE, and also for the Helmholtz PDE, by using the results mentioned above.

\end{document}